\documentclass[12pt]{amsart}
\usepackage{amsmath,amssymb,amsthm}
\usepackage{mathtools}
\usepackage{enumitem}
\usepackage{hyperref}
\usepackage{subcaption}
\usepackage{IEEEtrantools}
\usepackage{cleveref}
\usepackage[T1]{fontenc}
\usepackage[utf8]{inputenc}
\usepackage{microtype}
\usepackage{tikz}
\usetikzlibrary{arrows.meta, positioning, calc}

\newtheorem{theorem}{Theorem}[section]
\newtheorem{lemma}[theorem]{Lemma}
\newtheorem{proposition}[theorem]{Proposition}
\newtheorem{corollary}[theorem]{Corollary}

\theoremstyle{definition}

\newtheorem{remark}[theorem]{Remark}

\newcommand{\OtG}{O_t(G)}

\newcommand{\OAztG}{O^{A_0}_t(G)}
\newcommand{\OAztK}{O^{A_0}_t(K)}

\newcommand{\RAzp}{R^{A_0}_+}
\newcommand{\RAzm}{R^{A_0}_-}
\newcommand{\Ut}{U_t(\mathfrak{g})}

\newcommand{\Aform}{A_0}
\newcommand{\Qpt}{\mathbb{Q}(t)}
\newcommand{\g}{\mathfrak{g}}
\newcommand{\Pp}{P^+}
\newcommand{\wt}{\mathrm{wt}}

\newcommand{\CpKo}{C(K_0)}

\newcommand{\HSoi}{\mathcal{H}_{\mathrm{Soi}}}
\newcommand{\psiqSoi}{\psi^{(q)}_{\mathrm{Soi}}}

\title[Triangular Decomposition of the Crystal Lattice Revisited]{%
  Triangular Decomposition of the Crystal Lattice\\
  of Quantized Function Algebras: Revisited}

\author{Ayan Dey}
\address{Indian Statistical Institute, Delhi, India}
\email{studentayandey@gmail.com, ayan22r@isid.ac.in}



\begin{document}

\begin{abstract}
Let $\g$ be a simple complex Lie algebra of type $G_2$, $F_4$, or $E_8$, and let $G$ be
the unique connected simply connected complex Lie group with $\mathrm{Lie}(G)=\g$ and compact real form $K$.
We prove a triangular decomposition theorem for the lower crystal lattice
$\OAztG$ of the quantized function algebra $\OtG$, establishing that
$\OAztG=A_0\text{-alg}<\RAzp \cup \RAzm>.$
This extends the triangular decomposition recently obtained for types
$A_n, B_n, C_n, D_n, E_6$, and $E_7$ in~\cite{DDPa} to all simple complex Lie algebras.
As a consequence, we obtain: (i) the inclusion
$\OAztG\subseteq\OAztK$ conjectured by Matassa-Yuncken and
(ii) the crystal limit $\CpKo$ is a compact quantum
semigroup with a unique bi-invariant (Haar) state.
\end{abstract}

\maketitle

\section{Introduction}

The theory of crystal bases, introduced by Kashiwara~\cite{Kas90,Kas91,Kas93},
provides a powerful combinatorial tool for the study of representations of
quantized universal enveloping algebras.
A parallel and equally important object is the quantized function algebra
$\OtG$, whose crystal lattice $\OAztG$ captures the \emph{crystallization}
of the algebra of functions on the Lie group $G$ as the deformation parameter
$t\to 0$.

The notion of crystallization of a quantized function algebra on the $C^*$-algebraic context was first
introduced by Giri and Pal~\cite{GP24} for type $A_n$, and subsequently
extended to all complex semisimple Lie algebras by Matassa and
Yuncken~\cite{MY23}, in a remarkable way, by exploiting the theory of crystal bases to construct
the crystallized algebras as higher-rank graph $C^*$-algebras.
The rank of these graph algebras equals the rank of the underlying Lie algebra, an outstanding structural result.

One of the central structural questions in the theory is the
\emph{triangular decomposition} of the crystal lattice, which asserts that
\begin{equation}\label{eq:tridecomp}
  \OAztG=A_0\text{-alg}<\RAzp \cup \RAzm>.
\end{equation}
Here $\RAzp$ and $\RAzm$ are the $\Aform$-subalgebras of $\OAztG$ defined in \Cref{+- algebras}, and these serve as crystal analogs of the positive
and negative parts in the triangular decomposition of the quantized function algebra.

In the companion paper~\cite{DDPa}, Das, Dey, and Pal established
\eqref{eq:tridecomp} for all simple complex Lie algebras of types
$A_n, B_n, C_n, D_n, E_6$, and $E_7$.
The argument in~\cite{DDPa} relies crucially on the existence of a
\emph{minuscule} dominant weight as a tensor generator for the representation
ring.
For types $G_2$, $F_4$, and $E_8$, no minuscule dominant weight exists;
the tensor generator is instead the \emph{quasi-minuscule} fundamental
weight $\varpi_i$ (with $i=1, 4, 8$ for $G_2, F_4, E_8$ respectively).
The quasi-minuscule representation has a non-trivial weight space of weight 0,
and this is precisely where the argument of~\cite{DDPa} breaks down for these three exceptional types.

The present article fills this gap.
The key new insight is Lemma \ref{lem:quasi_min_crystal}, which gives a detailed
description of the weight-zero nodes in the connected component of
highest weight $\varpi_i$ inside the tensor product crystal
$B(\varpi_i)\otimes B(\varpi_i)$.
We show that each such node is of the form $b'\otimes c'$ where
$-\wt(b')=\wt(c')<0$, which allows us to bypass the obstruction
identified in~\cite{DDPa}.\\
A priori, the crystal analog of triangular decomposition may seem like an isolated result in crystal basis theory, but in the final section, we also indicate how this completely Lie theoretic result can be used to answer several operator algebraic questions related to the crystallized algebra of Matassa and Yuncken.

\subsection*{Main results}
Let $G$ be any connected simply connected complex Lie group $G$ with complex simple Lie algebra $\mathfrak{g}$ and compact real form $K$.
\begin{enumerate}[label=(\roman*)]
\item
  \textbf{Triangular decomposition} (Theorem \ref{thm:tridecomp_exceptional}):
   The crystal lattice satisfies
  $\OAztG= A_0\text{-alg}<\RAzp \cup \RAzm>$.
\item
  \textbf{The inclusion}
  (Corollary \ref{cor:MY_conjecture}):
  $\OAztG \subseteq \OAztK$.
\item
  \textbf{Compact quantum semigroup structure}
  (Corollary \ref{cor:compact_qsg}):
  The crystal limit $\CpKo$ is a compact quantum semigroup with a unique bi-invariant (Haar) state.
\end{enumerate}
\subsection*{Organization}

\Cref{sec:prelim} fixes notation and recalls the necessary background on
crystal bases, the crystal lattice of $\OtG$, and the spaces $\RAzp$,
$\RAzm$.
\Cref{sec:quasi_min_crystal} proves the key crystal-theoretic lemma on the
quasi-minuscule crystal.
\Cref{sec:main} contains the proof of the triangular decomposition for the
exceptional types.
\Cref{sec:consequences} states the consequences regarding the inclusion $\OAztG \subseteq \OAztK$ and the compact quantum semigroup structure with a Haar state.

\section{Prerequisites}\label{sec:prelim}

\subsection{Lie-theoretic conventions}

Throughout, $\g$ denotes a simple Lie algebra of rank $n$ with
Cartan matrix $(a_{i,j})$ and symmetrizer $\mathrm{diag}(d_1,\dots,d_n)$,
where each $d_i=\tfrac{1}{2}(\alpha_i,\alpha_i)$ is a positive integer and $d_i\in \{1,2,3\}$. Let $\alpha_1,\dots,\alpha_n$ be the simple roots,
$\alpha_1^\vee,\dots,\alpha_n^\vee$ the coroots, and
$\varpi_1,\dots,\varpi_n$ the fundamental weights.
The set of dominant integral weights is $\Pp$; we write $\rho$ for the
sum of fundamental weights.
For $\mu=\sum_i m_i\alpha_i$, set $K_\mu = K_1^{m_1}\cdots K_n^{m_n}$.

\subsection{The quantized enveloping algebra}

The quantized universal enveloping algebra $\Ut$ is the $\Qpt$-algebra
generated by $K_i, K_i^{-1}, E_i, F_i$ ($1\le i\le n$) subject to the
standard relations; see~[\cite{DDPa}] for the
precise form we use.
The Hopf structure and compact real form are as in~\cite[Section~2.1]{DDPa}.
We work over the base field $\Qpt$ (rather than $\mathbb{C}(t)$) so that
the results of~\cite{Jan-1996aa, Joseph95, Kas91} apply directly.

\subsection{Crystal bases and global bases}

We use $\Aform = \{f(t)/g(t)\in\Qpt : g(0)\ne 0\}$ for the localization
of $\mathbb{Q}[t]$ at zero.

For each $\Lambda\in\Pp$, let $V(\Lambda)$ be the irreducible highest
weight $\Ut$-module with highest weight vector $v^\Lambda_1$ and $m_\Lambda:=\text{dim}_{\mathbb{Q}(t)}V(\Lambda)$.
We fix a crystal basis $(L(\Lambda), B(\Lambda))$, polarization $(\cdot,\cdot)$,
and global crystal basis or canonical basis $\{v^\Lambda_i : 1\le i\le \dim\Lambda\}$.
The crystal basis
$B(\Lambda)$ then consists of the elements $\{b^\Lambda_i:=v^\Lambda_i + tL(\Lambda): 1\le i\le \dim\Lambda\}$. Subsequently, the functional $(v^\Lambda_i, -)$ on $V(\Lambda)$ is denoted by $(v^\Lambda_i)^*$. 

\subsection{The quantized function algebra and its crystal lattice}

The quantized function algebra $\OtG$ is the $\Qpt$-span of all matrix
elements
\[
  C^\Lambda_{f,v}(a) = \langle f, av\rangle,
  \quad \Lambda\in\Pp,\; f\in V(\Lambda)^*,\; v\in V(\Lambda),\; a\in\Ut.
\]
It carries a $\Ut\otimes\Ut$-module structure via left and right regular
actions, yielding the direct sum decomposition
\[
  \OtG \;\cong\; \bigoplus_{\Lambda\in\Pp} V(\Lambda)\otimes V(\Lambda).
\]
The \emph{lower crystal lattice} is the $\Aform$-submodule
\[
  \OAztG \;\cong\; \bigoplus_{\Lambda\in\Pp}
  L(\Lambda)\otimes_{\Aform} L(\Lambda),
\]
which in terms of the global basis reads
\[
  \OAztG \;=\; \Aform\text{-span of }
  \bigl\{C^\Lambda_{(v^\Lambda_i)^*,v^\Lambda_j} :
  \Lambda\in\Pp,\; 1\le i,j\le \dim\Lambda\bigr\}.
\]
We write $C^\Lambda_{i,j}$ for $C^\Lambda_{(v^\Lambda_i)^*,v^\Lambda_j}$.

\subsection{The algebras $\RAzp$ and $\RAzm$}
\label{+- algebras}

Following~\cite[Section~3.1]{DDPa} we define:
\begin{align*}
  \RAzp := \Aform\text{-span of }
    \bigl\{C^\Lambda_{i,1} : \Lambda\in\Pp,\;1\le i\le \dim\Lambda\bigr\}, \\
  \RAzm := \Aform\text{-span of }
    \bigl\{C^\Lambda_{i,m_\Lambda} :
    \Lambda\in\Pp,\;1\le i\le \dim\Lambda\bigr\},\\
  \tilde{R}^{A_0}_- := \Aform\text{-span of }
    \bigl\{t^{(w_0\Lambda-\wt(v^{\Lambda}_i),\;\rho)}C^\Lambda_{i,m_\Lambda} :
    \Lambda\in\Pp,\;1\le i\le \dim\Lambda\bigr\}.
\end{align*}
where $v^\Lambda_1$, $v^\Lambda_{m_\Lambda}$
are the highest and lowest weight global basis vectors of weights $\Lambda$ and $w_0.\Lambda$, respectively. In fact, it also follows that $\RAzp$ and $\RAzm$ are $\Aform$-algebras.\\
The $\Aform$-subalgebra $\OAztK$ of $\OtG$ is defined to be the $\Aform$-algebra generated
by $R^{A_0}_+$ and $(R^{A_0}_+)^*$. Therefore, the conjectured inclusion $\OAztG\subseteq\OAztK$
can be obtained from the triangular decomposition \eqref{eq:tridecomp} together
with the $\ast$-structure relation $(R^{A_0}_+)^* = \tilde{R}^{A_0}_-$ and ${R}^{A_0}_-\subseteq \tilde{R}^{A_0}_-$.

\subsection{The quasi-minuscule representations}\label{subsec:quasimin}

A non-zero dominant weight $\Omega\in\Pp$ is \emph{minuscule} if every weight of
$V(\Omega)$ lies in the Weyl group orbit of $\Omega$.
It is \emph{quasi-minuscule} if the only weight of $V(\Omega)$ not in the orbit $W\cdot\Omega$ is the zero weight. Each simple Lie algebra has a unique quasi-minuscule highest weight module, and the multiplicity of the zero weight is the number of short simple roots. The highest weight of that quasi-minuscule module is the highest short root of the associated root system, which in the simply-laced case is also the highest long root, making it the adjoint module. Therefore, the set of non-zero weights of the quasi-minuscule highest weight module coincides with the set of all short roots, and these non-zero weights appear with multiplicity 1. For an explicit construction of the quantum analog of the quasi-minuscule module, see \cite[5A.2]{Jan-1996aa}. We assume the Bourbaki numbering of the Dynkin diagram. For types $G_2$, $F_4$, and $E_8$, no minuscule weight exists. The quasi-minuscule weight for these three exceptional types is, in fact, a fundamental weight $\varpi_i$ (see  Appendix~A.2.3, \cite{LakRag-2008aa}):
\begin{center}
\begin{tabular}{ccc}
  Type & Index $i$ & $\varpi_i$ \\ \hline
  $G_2$ & $1$ & $\dim=7$ \\
  $F_4$ & $4$ & $\dim=26$ \\
  $E_8$ & $8$ & $\dim=248 (\text{adjoint})$
\end{tabular}
\end{center}
In each case, $\varpi_i$ is a tensor generator for the representation ring;
that is, every irreducible $V(\Lambda)$ is isomorphic to a
direct summand of some tensor power $V(\varpi_i)^{\otimes r}$
(see~\cite[Prop. 5A.10]{Jan-1996aa}).
The multiplicity of the zero weight in $V(\varpi_i)$ is $m = 1, 2, 8$ for
$G_2, F_4, E_8$ respectively. Furthermore, it follows from ~\cite[Lemma 5A.9.b]{Jan-1996aa} that in either of these three types, $V(\varpi_i)$ appears with multiplicity 1 as a direct summand in the decomposition of $V(\varpi_i)\otimes V(\varpi_i)$.

\subsection{Obstruction in the quasi-minuscule case}\label{subsec:obstruction}

The proof of the crystal analog of triangular decomposition for the types $A_n, B_n, C_n, D_n, E_6, \text{and}\; E_7$ relies on \cite[Theorem~3.6]{DDPa}, stated below. 

\begin{theorem}
\label{DDPm}
Let $\mathfrak{g}$ be a simple complex Lie algebra.
For any dominant weight $\Omega\neq 0$, choose a lower global basis 
vector $v^{\Omega}_{j}$ of $V(\Omega)$ with $\wt(v^{\Omega}_{j})\neq 0$. 
Then the following are equivalent.
\begin{enumerate}
\item  
There are dominant weights $\Lambda,\Gamma$, and a $U_t(\mathfrak{g})$-module morphism
\begin{center}
 $T: V(\Lambda)\otimes V(-\omega_{0}.\Gamma)\to V(\Omega)$ 
\end{center}
        such that $T( L(\Lambda)\otimes L(-\omega_{0}.\Gamma))=L(\Omega)$ and 
        $T(v^{\Lambda}_{\Lambda}\otimes v^{-\omega_{0}.\Gamma}_{-\Gamma})=v^{\Omega}_{j}$.\\
\item 
$\wt(v^{\Omega}_{j})=\omega. \Omega$ for some $\omega\in W$.
\end{enumerate}
\end{theorem}
The non-zero weight condition of Theorem \ref{DDPm} is essential: for global basis vectors of weight zero, it breaks down as indicated in the next Theorem
(see \cite[Proposition~3.10]{DDPa}). 

\begin{theorem}
Let $\mathfrak{g}$ be a simple complex Lie algebra and let $\Lambda,\Omega\in P_{+}$ with $\Omega \neq 0$.
Let $v^{\Omega}_{j}$ be a nonzero global basis vector in $V(\Omega)$ of weight 0. Then there does not exist any $U_t(\mathfrak{g})$-module morphism
\[
T: V(\Lambda)\otimes V(-w_0.\Lambda) \to V(\Omega)
\]
such that $T(L(\Lambda)\otimes L(-\omega_{0}.\Lambda))=L(\Omega)$ and 
$T(v_{\Lambda}^{\Lambda}\otimes v_{-\Lambda}^{-\omega_{0}.\Lambda})=v^{\Omega}_{j}$.
\label{obs}
\end{theorem}
It is precisely these weight-zero global basis vectors that must be handled
separately in the quasi-minuscule case.

\section{A Crystal Lemma for the Quasi-Minuscule Representation}%
\label{sec:quasi_min_crystal}

\subsection{The \texorpdfstring{$G_2$}{G2} case: crystal graphs}\label{subsec:G2_crystal}
Before giving the general proof, we illustrate Lemma \ref{lem:quasi_min_crystal}
concretely in the $G_2$ case, where everything can be drawn explicitly. We use the Bourbaki labeling: $\alpha_1$(short) and $\alpha_2$(long) are two simple roots.\\
\tikzset{
  dnode/.style={circle, draw, fill=white, inner sep=0pt, minimum size=6pt},
  every label/.append style={font=\scriptsize}
}
\begin{center}
\begin{tikzpicture}
  \node[dnode, label=below:$1$] (1) at (0,0) {};
  \node[dnode, label=below:$2$] (2) at (1,0) {};
  
  \draw (1) -- (2);
  \draw[transform canvas={yshift=2.5pt}] (1) -- (2);
  \draw[transform canvas={yshift=-2.5pt}] (1) -- (2);
  \draw (0.55, 0.15) -- (0.4, 0) -- (0.55, -0.15); 
\end{tikzpicture}
\end{center}
The quasi-minuscule fundamental
weight is then $\varpi_1=2\alpha_1+\alpha_2$ with $\langle\alpha_1^\vee,\varpi_1\rangle=1$ and
$\langle\alpha_2^\vee,\varpi_1\rangle=0$.
The representation $V(\varpi_1)$ is $7$-dimensional; its weights, listed from
top to bottom, are
\[
  \varpi_1,\quad
  \varpi_1{-}\alpha_1,\quad
  \varpi_1{-}\alpha_1{-}\alpha_2,\quad
  0,\quad
  {-}(\varpi_1{-}\alpha_1{-}\alpha_2),\quad
  {-}(\varpi_1{-}\alpha_1),\quad
  {-}\varpi_1.
\]
We label the corresponding crystal basis elements $b_1,\dots,b_7$ in this order.
\Cref{fig:G2_crystal} shows the crystal graph $B(\varpi_1)$: a straight chain
in which $\tilde{f}_1$ and $\tilde{f}_2$ alternate (the edge from $b_3$ to $b_4$
and from $b_4$ to $b_5$ are both labelled $1$).

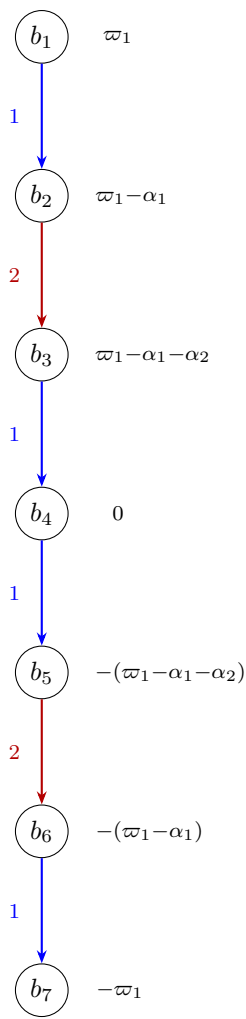
\begin{figure}[ht]
\centering
\begin{tikzpicture}[
    node distance = 1.4cm,
    every node/.style = {draw, circle, minimum size=7mm, inner sep=1pt,
                         font=\small},
    arr/.style = {-{Stealth[length=5pt]}, thick},
    f1/.style  = {arr, blue},
    f2/.style  = {arr, red!70!black}
  ]
  \node (b1) {$b_1$};
  \node (b2) [below=of b1] {$b_2$};
  \node (b3) [below=of b2] {$b_3$};
  \node (b4) [below=of b3] {$b_4$};
  \node (b5) [below=of b4] {$b_5$};
  \node (b6) [below=of b5] {$b_6$};
  \node (b7) [below=of b6] {$b_7$};

  \node[draw=none, right=0.3cm of b1, font=\scriptsize] {$\varpi_1$};
  \node[draw=none, right=0.3cm of b2, font=\scriptsize] {$\varpi_1{-}\alpha_1$};
  \node[draw=none, right=0.3cm of b3, font=\scriptsize] {$\varpi_1{-}\alpha_1{-}\alpha_2$};
  \node[draw=none, right=0.3cm of b4, font=\scriptsize] {$0$};
  \node[draw=none, right=0.3cm of b5, font=\scriptsize] {${-}(\varpi_1{-}\alpha_1{-}\alpha_2)$};
  \node[draw=none, right=0.3cm of b6, font=\scriptsize] {${-}(\varpi_1{-}\alpha_1)$};
  \node[draw=none, right=0.3cm of b7, font=\scriptsize] {${-}\varpi_1$};

  \draw[f1] (b1) -- node[draw=none,left,font=\scriptsize,blue]{$1$} (b2);
  \draw[f2] (b2) -- node[draw=none,left,font=\scriptsize,red!70!black]{$2$} (b3);
  \draw[f1] (b3) -- node[draw=none,left,font=\scriptsize,blue]{$1$} (b4);
  \draw[f1] (b4) -- node[draw=none,left,font=\scriptsize,blue]{$1$} (b5);
  \draw[f2] (b5) -- node[draw=none,left,font=\scriptsize,red!70!black]{$2$} (b6);
  \draw[f1] (b6) -- node[draw=none,left,font=\scriptsize,blue]{$1$} (b7);
\end{tikzpicture}
\caption{The crystal graph $B(\varpi_1)$ for $G_2$. Blue arrows are
$\tilde{f}_1$-edges, red arrows are $\tilde{f}_2$-edges. The weight-zero node is $b_4$. (See Proposition 2.1, \cite{KM94})}
\label{fig:G2_crystal}
\end{figure}

\medskip

Now consider the tensor product crystal $B(\varpi_1)\otimes B(\varpi_1)$.
We apply the tensor product rule: $\tilde{f}_k$ acts on the \emph{right}
factor of $b\otimes c$ if $\varphi_k(b)\le\varepsilon_k(c)$, and on the
\emph{left} factor if $\varphi_k(b)>\varepsilon_k(c)$.
From the chain structure of $B(\varpi_1)$ one reads off the string lengths:
\[
\begin{array}{c|ccccccc}
  & b_1 & b_2 & b_3 & b_4 & b_5 & b_6 & b_7 \\ \hline
\varphi_1 & 1 & 0 & 2 & 1 & 0 & 1 & 0 \\
\varepsilon_1 & 0 & 1 & 0 & 1 & 2 & 0 & 1 \\
\varphi_2 & 0 & 1 & 0 & 0 & 1 & 0 & 0 \\
\varepsilon_2 & 0 & 0 & 1 & 0 & 0 & 1 & 0
\end{array}
\]
The unique highest weight node of weight $\varpi_1$ in
$B(\varpi_1)\otimes B(\varpi_1)$ is $x_1 = b_1\otimes b_4$,
since $\wt(b_1)+\wt(b_4)=\varpi_1$.
Applying the tensor product rule step by step yields the seven nodes of
$\mathcal{C}$:
\begin{align*}
  x_1 &= b_1\otimes b_4, & \wt &= \varpi_1, \\
  x_2 &= b_1\otimes b_5, & \wt &= \varpi_1{-}\alpha_1, \\
  x_3 &= b_1\otimes b_6, & \wt &= \varpi_1{-}\alpha_1{-}\alpha_2, \\
  x_4 &= b_2\otimes b_6, & \wt &= 0, \\
  x_5 &= b_2\otimes b_7, & \wt &= {-}(\varpi_1{-}\alpha_1{-}\alpha_2), \\
  x_6 &= b_3\otimes b_7, & \wt &= {-}(\varpi_1{-}\alpha_1), \\
  x_7 &= b_4\otimes b_7, & \wt &= {-}\varpi_1,
\end{align*}
with crystal edges
$x_1\xrightarrow{1}x_2\xrightarrow{2}x_3\xrightarrow{1}x_4
\xrightarrow{1}x_5\xrightarrow{2}x_6\xrightarrow{1}x_7$.
For the reader's convenience we spell out the rule at each step:
\begin{itemize}
  \item $x_1\to x_2$: $\varphi_1(b_1)=1\le\varepsilon_1(b_4)=1$,
        right: $b_1\otimes\tilde{f}_1(b_4)=b_1\otimes b_5$.
  \item $x_2\to x_3$: $\varphi_2(b_1)=0\le\varepsilon_2(b_5)=0$,
        right: $b_1\otimes\tilde{f}_2(b_5)=b_1\otimes b_6$.
  \item $x_3\to x_4$: $\varphi_1(b_1)=1>\varepsilon_1(b_6)=0$,
        left: $\tilde{f}_1(b_1)\otimes b_6=b_2\otimes b_6$.
  \item $x_4\to x_5$: $\varphi_1(b_2)=0\le\varepsilon_1(b_6)=0$,
        right: $b_2\otimes\tilde{f}_1(b_6)=b_2\otimes b_7$.
  \item $x_5\to x_6$: $\varphi_2(b_2)=1>\varepsilon_2(b_7)=0$,
        left: $\tilde{f}_2(b_2)\otimes b_7=b_3\otimes b_7$.
  \item $x_6\to x_7$: $\varphi_1(b_3)=2>\varepsilon_1(b_7)=1$,
        left: $\tilde{f}_1(b_3)\otimes b_7=b_4\otimes b_7$.
\end{itemize}
The right tensor factors are $b_4,b_5,b_6,b_6,b_7,b_7,b_7$
with weights $0,{<}0,{<}0,{<}0,{<}0,{<}0,{<}0$ respectively.
Hence for every $k\ge 2$ the right factor of $x_k$ has strictly negative
weight, confirming Lemma \ref{lem:quasi_min_crystal}(i).
The unique weight-zero node is $x_4=b_2\otimes b_6$, where
$\wt(b_2)=\varpi_1-\alpha_1>0$ and $\wt(b_6)=-(\varpi_1-\alpha_1)<0$,
so $\beta_1=\varpi_1-\alpha_1$ is the positive root in the notation of Theorem
\ref{thm:tridecomp_exceptional}, confirming Lemma \ref{lem:quasi_min_crystal}(ii).
\Cref{fig:G2_subcrystal} displays this subcrystal.

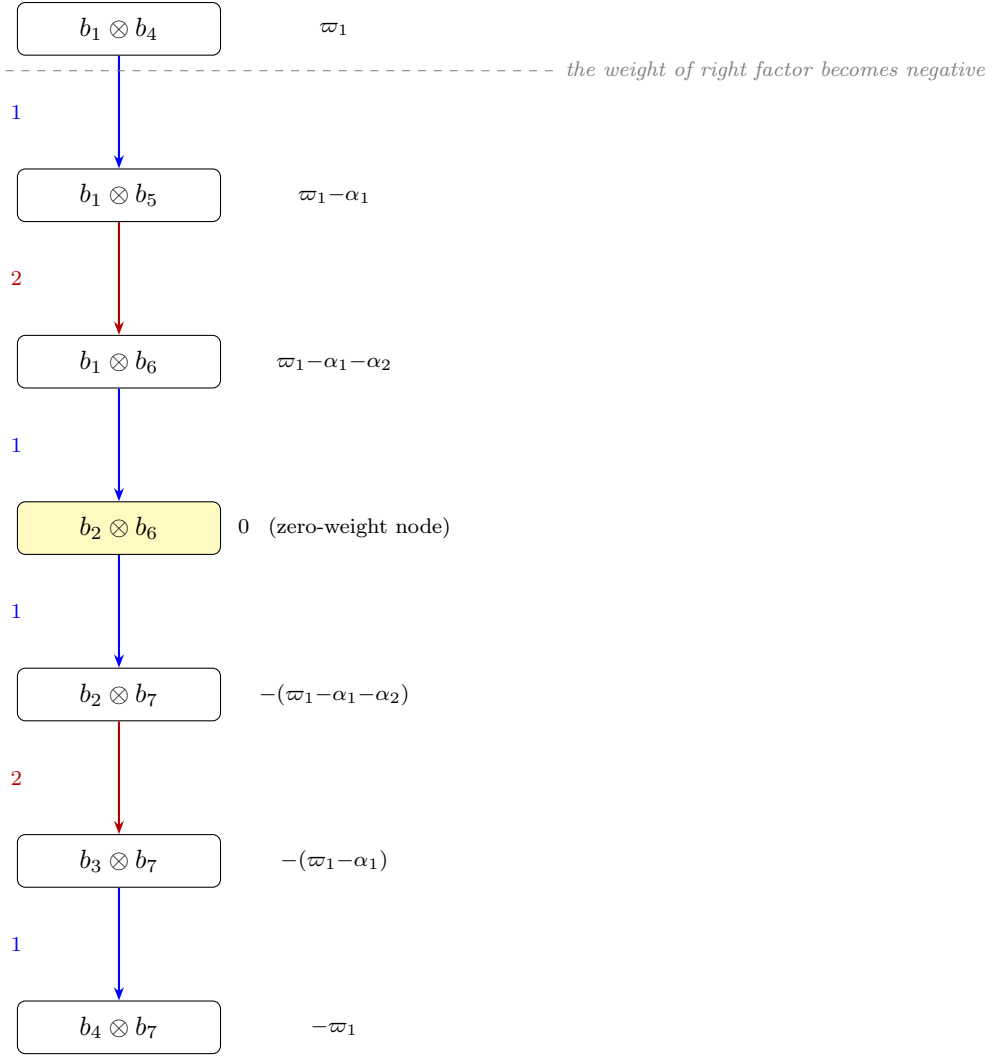
\begin{figure}[ht]
\centering
\begin{tikzpicture}[
    node distance = 1.5cm,
    every node/.style = {draw, rounded corners=3pt, minimum width=2.7cm,
                         minimum height=7mm, inner sep=2pt, font=\small},
    arr/.style = {-{Stealth[length=5pt]}, thick},
    f1/.style  = {arr, blue},
    f2/.style  = {arr, red!70!black},
    highlight/.style = {fill=yellow!30}
  ]
  \node (x1) {$b_1 \otimes b_4$};
  \node (x2) [below=of x1] {$b_1 \otimes b_5$};
  \node (x3) [below=of x2] {$b_1 \otimes b_6$};
  \node (x4) [below=of x3, highlight] {$b_2 \otimes b_6$};
  \node (x5) [below=of x4] {$b_2 \otimes b_7$};
  \node (x6) [below=of x5] {$b_3 \otimes b_7$};
  \node (x7) [below=of x6] {$b_4 \otimes b_7$};

  \node[draw=none, right=0.15cm of x1, font=\scriptsize]
       {$\varpi_1$};
  \node[draw=none, right=0.15cm of x2, font=\scriptsize]
       {$\varpi_1{-}\alpha_1$};
  \node[draw=none, right=0.15cm of x3, font=\scriptsize]
       {$\varpi_1{-}\alpha_1{-}\alpha_2$};
  \node[draw=none, right=0.15cm of x4, font=\scriptsize]
       {$0$ \enspace{\normalfont(zero-weight node)}};
  \node[draw=none, right=0.15cm of x5, font=\scriptsize]
       {${-}(\varpi_1{-}\alpha_1{-}\alpha_2)$};
  \node[draw=none, right=0.15cm of x6, font=\scriptsize]
       {${-}(\varpi_1{-}\alpha_1)$};
  \node[draw=none, right=0.15cm of x7, font=\scriptsize]
       {${-}\varpi_1$};

  \draw[f1] (x1) -- node[draw=none,left,font=\scriptsize,blue]{$1$} (x2);
  \draw[f2] (x2) -- node[draw=none,left,font=\scriptsize,red!70!black]{$2$} (x3);
  \draw[f1] (x3) -- node[draw=none,left,font=\scriptsize,blue]{$1$} (x4);
  \draw[f1] (x4) -- node[draw=none,left,font=\scriptsize,blue]{$1$} (x5);
  \draw[f2] (x5) -- node[draw=none,left,font=\scriptsize,red!70!black]{$2$} (x6);
  \draw[f1] (x6) -- node[draw=none,left,font=\scriptsize,blue]{$1$} (x7);

  \draw[dashed, gray] ($(x1.south west)+(-0.15,-0.2)$)
                   -- ($(x1.south east)+(4.5,-0.2)$)
    node[draw=none, right, font=\scriptsize\itshape, gray]
    {the weight of right factor becomes negative};
\end{tikzpicture}
\caption{The subcrystal $\mathcal{C}\cong B(\varpi_1)$ inside
$B(\varpi_1)\otimes B(\varpi_1)$ for $G_2$.
The yellow node $x_4=b_2\otimes b_6$ is the unique weight-zero node;
its left factor $b_2$ has weight $\varpi_1-\alpha_1>0$ and its right
factor $b_6$ has weight $-(\varpi_1-\alpha_1)<0$.
The right tensor factor is $b_4$ (weight~$0$) only at the top node
$x_1$; from $x_2$ onward every right factor has strictly negative weight.
Blue arrows are $\tilde{f}_1$-edges, red arrows are $\tilde{f}_2$-edges.}
\label{fig:G2_subcrystal}
\end{figure}

\medskip

\begin{lemma}\label{lem:quasi_min_crystal}
Let $\g$ be a simple Lie algebra and $\varpi_i$ be a fundamental weight. Assume that $V(\varpi_i)$ appears as a direct summand of $V(\varpi_i)\otimes V(\varpi_i)$ into irreducibles.
Let $x = b_{\varpi_i}\otimes c$ be a highest weight node of weight
$\varpi_i$ in the tensor product crystal $B(\varpi_i)\otimes B(\varpi_i)$,
where $b_{\varpi_i}$ is the highest weight node of $B(\varpi_i)$.
Let $\mathcal{C}$ be the connected crystal component generated by $x$.
Then:
\begin{enumerate}[label=\textup{(\roman*)}]
  \item For every node $y = b'\otimes c'\in\mathcal{C}$ with $y\ne x$, the
    weight of the right tensor factor satisfies $\wt(c') < 0$.
  \item If $y=b'\otimes c'\in\mathcal{C}$ is a weight-zero node, then
    $\wt(b') = -\wt(c') \ne 0$.
\end{enumerate}
\end{lemma}
\noindent In particular, if $\mathfrak{g}$ is either of the types $G_2, F_4, \text{or} \;E_8$, and $\varpi_i$ is the quasi-minuscule fundamental weight with $i=1,4,8$, the conclusion of Lemma \ref{lem:quasi_min_crystal} holds. 
\begin{proof}
Weight conservation gives $\wt(x) = \wt(b_{\varpi_i})+\wt(c)$.
Since $\wt(b_{\varpi_i})=\varpi_i$, we get $\wt(c)=0$.

\textit{Step 1: The lower Kashiwara operators $\tilde{f}_j$ for $j\ne i$ act trivialy on $x$.}
Since $b_{\varpi_i}$ is the highest weight element of $B(\varpi_i)$, its
$\varphi$ values equal the Dynkin labels:
$\varphi_j(b_{\varpi_i}) = \langle\alpha_j^\vee,\varpi_i\rangle = \delta_{ji}$.
The tensor product rule requires $\varepsilon_j(c)\le\varphi_j(b_{\varpi_i})$
for $x$ to be annihilated by $\tilde{e}_j$.
For $j\ne i$, $\varphi_j(b_{\varpi_i})=0$ and $\varepsilon_j(c)\ge 0$, so
$\varepsilon_j(c)=0$. Hence for all $j\neq i$, $\phi_j(c)=0$, and therefore it is easy to see that $\tilde{f}_j(x)= b_{\varpi_i}\otimes \tilde{f}_j(c)= 0$
for all $j\ne i$.

\textit{Step 2: $\varepsilon_i(c)\ge 1$.}
If $\varepsilon_i(c)=0$ then $\varepsilon_j(c)=0$ for all $j$, making $c$ a
highest weight element of $B(\varpi_i)$ of weight $0$.
However, $B(\varpi_i)$ contains no highest weight element of weight $0$ since
$\varpi_i$ is the unique highest weight of the irreducible crystal.
This contradiction establishes $\varepsilon_i(c)\ge 1$.

\textit{Step 3: First lowering step yields $\wt(c')<0$.}
From Step~1, $\tilde{f}_j(x)=0$ for $j\ne i$.
From Step~2, $\tilde{f}_i$ acts on the right factor:
$\tilde{f}_i(x) = b_{\varpi_i}\otimes \tilde{f}_ic\neq 0$.
The new right factor has weight $\wt(\tilde{f}_ic)=\wt(c)-\alpha_i=-\alpha_i<0$.

\textit{Step 4: All subsequent lowering steps preserve $\wt(c')<0$.}
Any node $y\in\mathcal{C}\setminus\{x\}$ is obtained from $\tilde{f}_i(x)$
by a sequence of operators $\tilde{f}_k$.
Each application of $\tilde{f}_k$ either acts on the left factor (leaving the
right factor unchanged) or subtracts a positive simple root $\alpha_k$ from
the weight of the right factor.
Starting from $\wt(c')=-\alpha_i<0$, the weight of the right factor remains
$\le -\alpha_i < 0$ under the root partial order.
This establishes part~(i).

Part~(ii) follows immediately: if $\wt(b')+\wt(c')=0$ and $\wt(c')<0$,
then $\wt(b')=-\wt(c')>0$, which is nonzero.
\end{proof}

\section{The Triangular Decomposition}%
\label{sec:main}
We record for reference the generating set result that will be used in the proof of the triangular decomposition.

\begin{proposition}[{\cite[Proposition~3.8]{DDPa}}]\label{prop:R_algebra_gen}
Let $\mathcal{F}\subset\Pp$ be a set of highest weights such that every
$V(\Lambda)$ ($\Lambda\in\Pp$) is a direct summand of some tensor product
$\bigotimes_{k=1}^r V(\Omega_k)$ with $\Omega_k\in\mathcal{F}$.
Then $\OAztG$ is the $\Aform$-algebra generated by
$\{C^\Omega_{v^{\Omega}_i,v^{\Omega}_j} : \Omega\in\mathcal{F},\; 1\le i,j\le\dim\Omega\}$.
\end{proposition}
We now prove the main theorem.

\begin{theorem}\label{thm:tridecomp_exceptional}
Let $\g$ be a simple complex Lie algebra of type $G_2$, $F_4$, or $E_8$,
and let $G$ be the connected simply connected complex Lie group with $\mathrm{Lie}(G)=\g$.
Then
 \[\OAztG = A_0\text{-alg}<\RAzp \cup \RAzm>\].
\end{theorem}
\begin{proof}
The inclusion $A_0\text{-alg}<\RAzp \cup \RAzm>\subseteq\OAztG$ is immediate from the
definitions.
For the reverse inclusion, it suffices by Proposition \ref{prop:R_algebra_gen} to
show that $C^\Omega_{v^{\Omega}_r,v^{\Omega}_j}\in A_0\text{-alg}<\RAzp \cup \RAzm>$ for all $r,j$ when $\Omega=\varpi_i$
(the quasi-minuscule fundamental weight), as $V(\varpi_i)$ is a tensor
generator.

\medskip

\noindent
\textbf{Case 1: $\wt(v^{\varpi_i}_j)\ne 0$.}
In this case $\wt(v^{\varpi_i}_j)=\omega \cdot\varpi_i$ for some $\omega\in W$,
and Theorem \ref{DDPm} directly gives
$C^{\varpi_i}_{v^{\varpi_{i}}_r,v^{\varpi_{i}}_j}\in\RAzp\cdot\RAzm$ for all $r$.

\medskip\noindent
\textbf{Case 2: $\wt(v^{\varpi_i}_j) = 0$.}
Let $m$ be the multiplicity of the zero weight in $V(\varpi_i)$; we have
$m=1,2,8$ for types $G_2, F_4, E_8$ respectively.
Let $\{v^{\varpi_i}_{1,0},\dots,v^{\varpi_i}_{m,0}\}$ be the global basis vectors
of weight zero in $L(\varpi_i)_0$.
By the general theory of crystal bases, there exists a $\Ut$-module
isomorphism
\[
  S : V(\varpi_i)\otimes V(\varpi_i) \xrightarrow{\;\sim\;}
  \bigoplus_k V(\lambda_k)
\]
that restricts to an
$\Aform$-linear isomorphism on the crystal lattices and the induced map $\tilde{S}$ on the crystal limit sends the product crystal $B(\varpi_i)\otimes B(\varpi_i)$
bijectively onto $\bigsqcup_k B(\lambda_k)$.
Since $V(\varpi_i)$ appears among the summands $V(\lambda_k)$ with multiplicity 1, we define the $\Ut$-module morphism
\[
  T := \pi_{V(\varpi_i)}\circ S :
  V(\varpi_i)\otimes V(\varpi_i)\longrightarrow V(\varpi_i),
\]
where $\pi_{V(\varpi_i)}$ is the projection onto the $V(\varpi_i)$ summand. Then $T$ sends the product crystal lattice onto $L(\varpi_i)$.\\
Let $\mathcal{C}$ denote the connected component of $B(\varpi_i)\otimes
B(\varpi_i)$ corresponding to $V(\varpi_i)$ under $S$.
By Lemma \ref{lem:quasi_min_crystal}, the $m$ weight-zero nodes of $\mathcal{C}$
have the form $b^{\varpi_i}_{\beta_k}\otimes b^{\varpi_i}_{-\beta_k}$ where $\wt(b^{\varpi_i}_{\beta_k})=\beta_k>0$ and
$\wt(b^{\varpi_i}_{-\beta_k})=-\beta_k<0$, for distinct short roots $\beta_1,\dots,\beta_m \in W.\varpi_i$.
Under the crystal isomorphism $\mathcal{C}\cong B(\varpi_i)$ induced by $S$,
these nodes map to the $m$ zero-weight crystal basis elements of $B(\varpi_i)$ bijectively.
Therefore, for each $k=1,\dots,m$:
\[
  T\!\left(v^{\varpi_i}_{\beta_k}\otimes v^{\varpi_i}_{-\beta_k}\right)
  = v^{\varpi_i}_{k,0} + t\cdot x_k
\]
for some $x_k\in L(\varpi_i)_0$.
The collection $\{v^{\varpi_i}_{k,0}+t\cdot x_k : k=1,\dots,m\}$
reduces modulo $tL(\varpi_i)_0$ to a $\mathbb{Q}$-basis of
$L(\varpi_i)_0/tL(\varpi_i)_0$.
By Nakayama's Lemma over the local ring $A_0$, it is an $\Aform$-basis
of $L(\varpi_i)_0$.
In particular, each global basis element $v^{\varpi_i}_{j,0}$ can be written as
\begin{equation}\label{eq:vj0_decomp}
  v^{\varpi_i}_{j,0} = \sum_{k=1}^m a_{jk}\cdot
  \bigl(v^{\varpi_i}_{k,0}+t\cdot x_k\bigr), \quad a_{jk}\in\Aform.
\end{equation}
By linearity, it suffices to show
$C^{\varpi_i}_{v^{\varpi_i}_r,\,v^{\varpi_i}_{k,0}+t\cdot x_k}
\in \RAzp\cdot\RAzm\cdot\RAzp\cdot\RAzm$ for each $k$.

Using the $\Ut$-equivariance of $T$ and the definition of the transpose map
$T^{\mathrm{tr}}$, for any $a\in\Ut$:
\begin{align*}
  C^{\varpi_i}_{v^{\varpi_i}_r,\,v^{\varpi_i}_{k,0}+t\cdot x_k}(a)
  &= \bigl\langle (v^{\varpi_i}_r)^*, a\cdot
    T(v^{\varpi_i}_{\beta_k}\otimes v^{\varpi_i}_{-\beta_k})\bigr\rangle \\
  &= \bigl\langle T^{\mathrm{tr}}(v^{\varpi_i}_r)^*,\,
    a\cdot (v^{\varpi_i}_{\beta_k}\otimes v^{\varpi_i}_{-\beta_k})\bigr\rangle \\
  &= C^{V(\varpi_i)\otimes V(\varpi_i)}_{T^{\mathrm{tr}}(v^{\varpi_i}_r)^*,\,
    v^{\varpi_i}_{\beta_k}\otimes v^{\varpi_i}_{-\beta_k}}(a).
\end{align*}
Since $T$ is lattice-preserving, its transpose $T^{\mathrm{tr}}$ is
$\Aform$-linear on the dual lattice, so
\[
  T^{\mathrm{tr}}(v^{\varpi_i}_r)^* =
  \sum_{p,q} c_{pq}(t)\cdot
  (v^{\varpi_i}_p\otimes v^{\varpi_i}_q)^*, \quad c_{pq}(t)\in\Aform.
\]
Applying the tensor product formula for matrix elements:
\[
  C^{\varpi_i}_{v^{\varpi_i}_r,\,v^{\varpi_i}_{k,0}+t\cdot x_k}
  = \sum_{p,q} c_{pq}(t)\cdot
  C^{\varpi_i}_{v^{\varpi_i}_p,\,v^{\varpi_i}_{\beta_k}}\cdot C^{\varpi_i}_{v^{\varpi_i}_q,\,v^{\varpi_i}_{-\beta_k}}.
\]
Since $\beta_k$ and $-\beta_k$ are non-zero weights lying in $W\cdot\varpi_i$.
Applying Theorem \ref{DDPm}:
\[
  C^{\varpi_i}_{v^{\varpi_i}_p,\,v^{\varpi_i}_{\beta_k}}\in\RAzp\cdot\RAzm
  \quad\text{and}\quad
  C^{\varpi_i}_{v^{\varpi_i}_q,\,v^{\varpi_i}_{-\beta_k}}\in\RAzp\cdot\RAzm
  \quad \text{for all } p,q.
\]
Therefore
\[
  C^{\varpi_i}_{v^{\varpi_i}_r,\,v^{\varpi_i}_{k,0}+t\cdot x_k}
  \in \RAzp\cdot\RAzm\cdot\RAzp\cdot\RAzm.
\]
Combining with \eqref{eq:vj0_decomp} and the fact that $\RAzp$ and $\RAzm$
are $\Aform$-algebras, we conclude
\[
  C^{\varpi_i}_{v^{\varpi_i}_r,\,v^{\varpi_i}_{j,0}}\in \RAzp\cdot\RAzm\cdot\RAzp\cdot\RAzm
  \subseteq A_0\text{-alg}<\RAzp \cup \RAzm>.
\]
In both Case~1 and Case~2, $C^{\varpi_i}_{v^{\varpi_i}_r,v^{\varpi_i}_j} \in A_0\text{-alg}<\RAzp \cup \RAzm>$. Therefore, we must have
$\OAztG\subseteq A_0\text{-alg}<\RAzp \cup \RAzm>$.
\end{proof}

\begin{corollary}\label{cor:tridecomp_all}
The triangular decomposition $\OAztG=A_0\text{-alg}<\RAzp \cup \RAzm>$ holds for every
complex simple Lie algebra $\g$.
\end{corollary}

\begin{proof}
The cases $A_n$, $B_n$, $C_n$, $D_n$, $E_6$, $E_7$ are
\cite[Theorem~3.11]{DDPa}, and the cases $G_2$, $F_4$, $E_8$ are Theorem
\ref{thm:tridecomp_exceptional} above.
Together, these exhaust all complex simple Lie algebras.
\end{proof}

\section{Consequences}\label{sec:consequences}

\subsection{The inclusion $\OAztG \;\subseteq\; \OAztK$.}
Recall that $\OAztK$ is the $\Aform$-algebra generated
by $R^{A_0}_+$ and $(R^{A_0}_+)^* = \tilde{R}^{A_0}_-$. 
Using Corollary \ref{cor:tridecomp_all} together with the definition of $\tilde{R}^{A_0}_-$, one immediately deduces the
following inclusion conjectured by Matassa and Yuncken,
(see~\cite[Remark~3.4]{MY23}) in complete generality.

\begin{corollary}\label{cor:MY_conjecture}
Let $\g$ be any complex simple Lie algebra and $G$ the corresponding
connected simply connected complex Lie group with compact real form $K$. Then
\[
  \OAztG \;\subseteq\; \OAztK.
\]
Equivalently, $\OAztK$ is the $\ast$-algebra $\widehat{O}^{A_0}_t(G)$
generated by $\OAztG$.
\end{corollary}

The equality $\OAztK = \widehat{O}^{A_0}_t(G)$ (i.e.\ $\OAztK$ is the
$\ast$-algebra generated by $\OAztG$) is used in a critical way in the
proof of the compact quantum semigroup structure below, as well as to
compare the two different crystallization approaches for type $A_n$ simple Lie algebras
in~\cite[Section~5]{DDPa}.

\subsection{The crystallized algebra $C(K_0)$} Given any $a\in \mathcal{O}_t(K)$, it is possible to assign an element $\vartheta_q(a)\in \mathcal{O}(K_q)$, for sufficiently small $q>0$ via a process called specialization. Furthermore, this assignment is compatible with the Hopf structures in the following sense. Let $a$ and $b$ be two elements of
$\mathcal{O}_t(K)$, then one has, for a small enough $\delta>0$
and for $q\in (0,\delta)$,
\begin{IEEEeqnarray*}{rClrClrCl}
\vartheta_{q}(ab) &=& \vartheta_{q}(a)\vartheta_{q}(b),\qquad &
\vartheta_{q}(a+b) &=& \vartheta_{q}(a)+\vartheta_{q}(b),\qquad
& \vartheta_{q}(a^{*}) &=& (\vartheta_{q}(a))^{*},\\ (\vartheta_{q}\otimes \vartheta_{q})\circ \Delta_{t}(a) 
    &=& \Delta_{q}(\vartheta_{q}(a)),
    & ev_{q}\circ \epsilon_{t}(a) &=& \epsilon_{q}\circ \vartheta_{q}(a).
     &&&
\end{IEEEeqnarray*}
Recall that for any $q\in (0,1)$, $\mathcal{O}(K_q)$ is equipped with a structure of a CQG algebra in the sense of Dijkhuizen and Koornwinder \cite{DK94}. Therefore, using subsequent results of \cite{DK94}, it follows that $\mathcal{O}(K_q)$ is injectively embedded inside its universal $C^*$-completion $C(K_q)$. On $C(K_q)$, we then have the faithful (Soibelman) representation $\pi^{(q)}_{Soi}:=\int_{\mathbb{T}}^{\oplus}\pi^{(q)}_{w_0, t}\;dt$. Here, the direct integral is taken over the maximal torus $\mathbb{T}\subset K$ of all the irreducible representations of $C(K_q)$ corresponding to the longest word $w_0$ of the Weyl group $W$ of $K$. $dt$ stands for the normalized Haar measure on $\mathbb{T}$. For more details about the specialization map $\vartheta_q$ and the Soibelman representation $\pi^{(q)}_{Soi}$, we refer to section 4 of \cite{DDPa}.\\
The \textit{crystallized algebra} $\CpKo$ of Matassa and Yuncken~\cite{MY23} is
the $C^*$-subalgebra of $\mathcal{B}(\HSoi)$ generated by the operators
$\{\lim_{q\to 0^+}\pi_{Soi}^{(q)}\circ\vartheta_q(a) : a\in\OAztK\}$. 

\begin{remark}
\label{rem:CMY_all}
The following gives more clarity about what makes the Matassa-Yuncken crystallization a truly genuine choice of a $C^*$-algebra to be called the crystal limit of quantized compact Lie groups $C(K_q)$.
\begin{align*}
  C(K_0) &= C^*\text{-subalgebra of }\mathcal{B}(\HSoi)\text{ generated by }
     \Bigl\{\lim_{q\to 0^+}\psiqSoi\circ\vartheta_q(a)
     : a\in\widehat{O}^{A_0}_t(G)\Bigr\} \\
  &= C^*\text{-subalgebra of }\mathcal{B}(\HSoi)\text{ generated by }
     \Bigl\{\lim_{q\to 0^+}\psiqSoi\circ\vartheta_q(a)
     : a\in\OAztG\Bigr\}.
\end{align*}
\end{remark}

\subsection{Compact quantum semigroup structure}
\begin{corollary}\label{cor:compact_qsg}
Let $\g$ be any complex simple Lie algebra and $K$ the corresponding
connected simply connected compact Lie group with $\text{Lie}(K)_{\mathbb{C}}=\mathfrak{g}$.
Then the comultiplication $\Delta_t$ and counit $\varepsilon_t$ of $\OAztK$
induce unital $\ast$-homomorphisms
\[
  \Delta : \CpKo \longrightarrow \CpKo\otimes_{min}\CpKo
  \qquad\text{and}\qquad
  \varepsilon : \CpKo \longrightarrow \mathbb{C}
\]
making $\CpKo$ a compact quantum semigroup. Furthermore, $C(K_0)$ admits a natural bi-invariant (Haar) state. 
\end{corollary}
A somewhat ad hoc proof of the quantum semigroup structure without the existence of the Haar state for $C(K_0)$ is given in \cite{DDPa} by exploiting the generators and relations of $C(K_0)$ introduced by Matassa and Yuncken. Here, we present a more analytic and self-contained proof.
\begin{proof}
From corollary \ref{cor:MY_conjecture}, it is easy to see that for any $a\in \OAztK$, $\Delta_t(a)$ is a finite
sum with summands in $\OAztK$, and therefore $\lim_{q\to 0+}[(\pi^{(q)}_{Soi}\circ \vartheta_q) \otimes (\pi^{(q)}_{Soi}\circ \vartheta_q)]\circ \Delta_t(a)$ is a well-defined element in $\mathcal{O}(K_0)\otimes \mathcal{O}(K_0)$. Write $\pi_0(a)$ for the operator $\lim_{q\to0+} \pi^{(q)}_{Soi}\circ \vartheta_q(a)$ and define, 
\begin{align*}
\Delta_0(\pi_0(a)):= \lim_{q\to 0+}[(\pi^{(q)}_{Soi}\circ \vartheta_q) \otimes (\pi^{(q)}_{Soi}\circ \vartheta_q)]\circ \Delta_t(a) =  \lim_{q\to 0+}[(\pi^{(q)}_{Soi}\otimes \pi^{(q)}_{Soi})\circ \Delta_q]\circ \vartheta_q(a).
\end{align*}
The latter equality follows from compatibility of $\vartheta_q$ with the co-algebra structures. To check that $\Delta_0$ is well-defined, it is enough to observe that
\begin{align*}
    ||\Delta_0(\pi_0(a))||\leq \lim_{q\to0+}||\vartheta_q(a)||_{\infty}= \lim_{q\to0+}||\pi^{(q)}_{Soi}\circ\vartheta_q(a)||=||\pi_0(a)||. 
\end{align*}
Here, $||.||_{\infty}$ stands for the universal $C^{*}$-norm on $\mathcal{O}(K_q)\subseteq C(K_q)$ and the latter equality comes from the faithfulness of the Soibelman representation $\pi^{(q)}_{Soi}$. $\Delta_0$ can be uniquely extended to $C(K_0)$ as a unital $*$-homomorphism from $C(K_0)\to C(K_0)\otimes_{min} C(K_0)$. Coassotiativity of $\Delta_0$ follows from coassotiativity of $\Delta_t$ after passing through the limit. Similarly, one can prove the existence of a counit on $C(K_0)$. Therefore, $C(K_0)$ admits a structure of a compact quantum semigroup.\\
Now, we proceed to prove the existence of the Haar state on $C(K_0)$. For any $a\in \mathcal{O}^{A_0}_t(K)$, let us first prove that $\lim_{q\to 0+} h_q(\vartheta_q(a))$ exists. Using the definition of $\vartheta_q$ and properties of $h_q$ $(h_q(1)=1,\; h_q(C^\Lambda_{f,v})=0,\; 0\neq \Lambda\in P_+)$, it follows that for sufficiently small q, 
\begin{align*}
        h_q(\vartheta_q(a))= \frac{p_1(t)}{p_2(t)}|_{t=q}, \quad p_1(t), p_2(t)\in \mathbb{Q}[t].
\end{align*}
Also, $h_q$ being a state, forces to have
\begin{align}
     |h_q(\vartheta_q(a))|\leq ||\vartheta_q(a)||_{\infty}=||\pi^{(q)}_{Soi}\circ \vartheta_q(a)||. 
\end{align}
Therefore, $\lim_{q\to 0+} |h_q(\vartheta_q(a))|$ is finite; it again follows that $\lim_{q\to 0+} h_q(\vartheta_q(a))$ is also finite. Define, 
\begin{align*}
     h_0(\pi_0(a)):= \lim_{q\to 0+} h_q(\vartheta_q(a)).
\end{align*} 
From (3), one can easily deduce that $h_0$ is well defined on $\mathcal{O}(K_0)$, and $h_0$ can be further uniquely extended to $C(K_0)$ as a state. Bi-invariance of $h_0$:
\begin{align*}
     (id\otimes h_0)\Delta_0(X)=h_0(X).\;Id_{H_{Soi}}=(h_0\otimes id)\Delta_0(X)
\end{align*}
for any $X\in C(K_0)$, follows from the bi-invariance of $h_q$ after passing through the limit and the uniqueness of $h_0$ is a straightforward consequence of bi-invariance. 
\end{proof}
It is also worthwhile to observe that, unlike in the case for $C(K_q)$, the Haar state on $C(K_0)$ may not be faithful. For instance, take $K=SU(2)$, it follows from the formula of Haar state for $C(SU_q(2))$ (see Appendix A.1 \cite{Wor87}) that $h_0(\pi_0(\alpha)^*\pi_0(\alpha))=0$ despite $\pi_0(\alpha)\neq 0$. 
\begin{remark}
Although all the main results, including Corollary \ref{cor:tridecomp_all}, Corollary \ref{cor:MY_conjecture}, and Corollary \ref{cor:compact_qsg}, are stated for simple Lie algebras, it should be noted that these continue to be valid for the semisimple case as well. 
\end{remark}

\textbf{Acknowledgment}.
The author is in debt to his thesis supervisor, Prof.~Arup K. Pal, for his continuous support, numerous helpful insights, and hour-long discussions. The author is grateful to Prof. Bipul Saurabh for discussions during the Ganit Symposium on Quantum groups at IIT Gandhinagar and for encouraging me to settle the $G_2$ case first, which eventually led to the proof for the remaining types. He also acknowledges Prof. Jyotishman Bhowmick and Bappa Ghosh for the discussions and for hosting him at ISI Calcutta.


\end{document}